\documentclass[11pt,a4paper]{amsart}
\usepackage[french, ngerman, italian, english]{babel}
\usepackage{amsmath,amsthm, amssymb, amsfonts}
\usepackage{setspace}
\usepackage{anysize}
\usepackage{url}
\usepackage{bm}
 \usepackage{ulem}
\usepackage{color,graphicx}
\usepackage[colorlinks=false]{hyperref}
\usepackage{paralist}
\usepackage{verbatim}
\usepackage[utf8]{inputenc}
\usepackage[all]{xy}

\theoremstyle{plain}
\newtheorem{theorem}{\bf Theorem}[section]

\newtheorem*{theorem*}{Theorem}

\newtheorem{corollary}[theorem]{\bf Corollary}

\newtheorem*{conjecture*}{\bf Conjecture}
\theoremstyle{definition}

\theoremstyle{remark}
\newtheorem{example}[theorem]{\bf Example}

\theoremstyle{example}

\def\D{\Delta}

\def \dim{{\operatorname{dim}}}

\def \HF{{\operatorname{HF}}}

\def \height{{\operatorname{ht}}}

\def \pd{{\operatorname{pd}}}

\begin{document}
\title{On a conjecture by Kalai}
\author[G. Caviglia]{Giulio Caviglia}  \thanks{The work of the first author was supported by a grant from the
Simons Foundation (209661 to G.~C.). }
\address{Department of Mathematics, Purdue University 150 N. University Street, West Lafayette, IN 47907-2067}
\email{gcavigli@math.purdue.edu}
\author[A. Constantinescu]{Alexandru Constantinescu}
\address{Institut de Math\'ematiques, Universit\`e de Neuch\^atel B\^atiment UniMail, Rue Emile-Argand 11
2000 Neuch\^atel - Switzerland}
\email{alexandru.constantinescu@unine.ch}
\author[M. Varbaro]{Matteo Varbaro}
\address{Dipartimento di Matematica,
Universit\`a degli Studi di Genova, Via Dodecaneso 35, 16146, Italy}
\email{varbaro@dima.unige.it}
\thanks{The authors thank the
Mathematical Sciences Research Institute, Berkeley CA, where this work was done,
for support and hospitality during Fall 2012  }
\date{{\small \today}}
\maketitle

\begin{abstract} We show that monomial ideals generated in degree two satisfy a conjecture by Eisenbud, Green and Harris. In particular we give a partial answer to a conjecture of Kalai by proving that $h$-vectors of flag Cohen-Macaulay simplicial complexes are
$h$-vectors of Cohen-Macaulay balanced simplicial complexes.
\end{abstract}
\maketitle
\section{Introduction}

An unpublished conjecture of Gil Kalai, recently verified by Frohmader \cite{Fr}, states that for any flag simplicial complex $\D$ there exists a balanced simplicial complex $\Gamma$ with the same $f$-vector. Here a $(d-1)$-simplicial complex is balanced if you can use $d$ colors to label its vertices so that no face contains two vertices of the same colour.
Kalai's conjecture has also a second part which is still open: If $\D$ happens to be Cohen-Macaulay (CM), then $\Gamma$ is required to be CM as well. 

In this note we show that for any CM flag simplicial complex $\D$ there exists a CM balanced simplicial complex $\Gamma$ with the same $h$-vector.
Such a result has been proved in \cite[Theorem 3.3]{CV} under the additional assumption that $\D$ is vertex decomposable. Other recent developments concerning Kalai's conjecture and related topics can be found in \cite{CN,BV}. To this purpose we will show a stronger statement, Theorem \ref{main}, namely that the \textit{Eisenbud-Green-Harris conjecture} ({\bf EGH}) holds  for quadratic monomial ideals. 

Let $S=K[x_1,\dots,x_n]$ be a polynomial ring over a field $K$. The {\bf EGH} conjecture, in the general form,  states that for every homogeneous ideal $I$ of $S$ containing a regular sequence $f_1,\dots,f_r$ of degrees $d_1\leq \cdots \leq d_r$ there exists a homogeneous ideal $J\subseteq S$, with the same Hilbert function as $I$ (i.e. ${\HF}_I={\HF}_J$) and containing $x_1^{d_1},\dots,x_r^{d_r}.$ Furthermore, by a theorem of Clements and Lindst\"om \cite{CL}, the ideal $J$, when it exists, can be chosen to be the sum of the ideal $(x_1^{d_1},\dots,x_r^{d_r})$ and a lex-segment ideal of $S$. We refer to \cite{EGH1} and \cite{EGH2} for the original formulation of this conjecture. The only large classes for which the {\bf EGH} conjecture is known are: when  $f_1,\dots,f_r$ are Gr\"obner basis (by a deformation argument), when   $d_i> \sum_{j<i}(d_j-1)$ for all $i=2,\dots,r$ (\cite{CM}) and 
when each $f_i$ factors as product of linear forms (\cite[Corollary 4.3]{Ab} for the case $r=n$, and \cite{Ab} together with \cite[Proposition 10]{CM} for the general case).

\section{The result}

\begin{theorem}\label{main}
Let $I\subseteq S=K[x_1,\ldots,x_n]$ be a monomial ideal generated in degree $2,$ of $\height I =g$. 
There exists a monomial ideal $J \in S$, such that $(x_1^2,\ldots,x_g^2)\subseteq J$ and $$\HF_I=\HF_J.$$ Furthermore $J$ can be chosen with the same projective dimension as $I.$ 
\end{theorem}
\begin{proof} Since the Hilbert function and the projective dimension are invariant with respect to field extension, we can assume without loss of generality that $K$ is infinite
We will prove that $I$ contains a regular sequence of the form $x_1\ell_1,\ldots, x_g\ell_g$, where $\ell_i$ is a linear form for every $i\in [g]=\{1,\ldots, g\}$. Then we will infer the theorem by a result in \cite{Ab}. 

Without loss of generality we may assume that $(x_1,\ldots,x_g)$ is a minimal prime of $I$. Thus, we may decompose the degree 2 component of $I$ as
\[I_2 = x_1V_1 \oplus \ldots \oplus x_gV_g,\] 
where each $V_i$ is a linear space generated by indeterminates. 
Our goal is to find $g$ linear forms $\ell_i\in V_i$, such that:

{\center (*) for all $A\subseteq [g]$, the $K$-vector space $\langle x_i:i\in A\rangle +\langle\ell_i:i\in[g]\setminus A\rangle$ has dimension $g$.}

\vspace{3mm}
 
\noindent To see that (*) is equivalent to $x_1\ell_1,\ldots, x_g\ell_g$ being a $S$-regular sequence (from now on we will just write regular sequence for $S$-regular sequence), consider the following short exact sequence (where $C=(x_1\ell_1,\ldots, x_g\ell_g)$):
\[0\rightarrow K[x_1,\ldots ,x_n]/(C:\ell_g)(-1)\rightarrow K[x_1,\ldots ,x_n]/C\rightarrow K[x_1,\ldots ,x_n]/(C+(\ell_g))\rightarrow 0.\]
By induction on the number of variables, both the extremes of the above exact sequence are graded modules of Krull dimension $\leq n-g$. By looking at the Hilbert polynomials, $K[x_1,\ldots ,x_n]/C$ has dimension $\leq n-g$ too, that is possible only if $C$ is a complete intersection.

So we have to seek $\ell_i\in V_i$ satisfying (*). Let $A=\{i_1,\ldots ,i_a\}$ be a subset of $[g]$. We define $U_A \subseteq \prod_{i\in A}V_i$ to be the following set:
\[U_A = \{ (v_{i_1},\ldots ,v_{i_a}) \in \prod_{i\in A}V_i~:~ \langle v_{i_1},\ldots ,v_{i_a}\rangle + \langle x_j:j \in [g]\setminus A\rangle \textup{ has dimension $g$}\}.\]
As the condition of linear dependence is obtained by imposing certain determinantal relations to be zero, $U_A$ is a Zariski open set of $\prod_{i\in A}V_i.$ Thus the $\widetilde{U}_A$ below is a Zariski open set of $\prod_{i=1}^g V_i$ \[\widetilde{U}_A = U_A \times \prod_{i\in [g]\setminus A} V_i \subseteq \prod_{i=1}^g V_i.\]
This construction can be done for every $A \subseteq [g]$, and thus we can define the open set
\[U  = \bigcap_{A\subseteq[g]}\widetilde{U}_A \subset \prod_{i=1}^g V_i.\]
Any element $(\ell_1,\ldots ,\ell_g)\in U$ will  automatically satisfy (*), so our goal is to show that $U\neq \emptyset$.
As $ \prod_{i=1}^g V_i$ is irreducible, it is enough to show that all the open sets $U_A$'s are nonempty.
For any $A\subseteq[g]$ we have 

\begin{equation}
\dim_K\biggl(\sum_{i\in A} V_i + \sum_{j\in[g]\setminus A} \langle x_j\rangle\biggr) \ge g,\label{C1}
\end{equation}
otherwise $(\sum_{i\in A} V_i + \sum_{j\in[g]\setminus A} \langle x_j\rangle)$  would be a prime ideal containing $I$ of height $<g$. 

Given $A\subseteq[g]$, we  define a bipartite graph $G_A$.  The  vertex set of $G_A$ has the following partition: $V(G_A) = \{x_1,\ldots,x_n\} \cup \{1,\dots,g\}$, and  the edge set of $G_A$ is given by:
\[ \{x_i,j\}\in E(G_A) \iff \left\{\begin{array}{lcl} x_i \in V_j&,&\textup{if~} j \in A\\
i=j&,&\textup{if~} j \notin A
\end{array}
\right.\]
We fix $A$ and prove that $G_A$ satisfies the hypothesis of the Marriage Theorem. 
For a subset  $B\subseteq V(G_A)$, we denote by $N(B)$ the set of vertices adjacent to some vertex in $B$. 
Choose now  $B \subseteq \{1,\ldots,g\}$.
By applying \eqref{C1} to the set $A\cap B \subseteq [g]$ we can deduce that 
\[\dim_K\biggl(\sum_{i\in A\cap B} V_i + \sum_{j\in([g]\setminus A)\cap B} \langle x_j\rangle\biggr) \ge |B|,\]
 
Furthermore notice that the dimension of the above vector space is $|N(B)|$, thus we can apply the Marriage Theorem and infer the existence of a matching in $G_A$ of the form $\{x_{i_j},j\}_{j\in [g]}$. Therefore $U_A$ is nonempty as it contains  $(x_{i_j} : j\in A).$

So we found a regular sequence of quadrics $f_1,\ldots ,f_g$ in $I$ consisting of products of linear forms.

Let $\pd(I)$ be the projective dimension of $I$ and assume that $\pd(I)=p-1$. By applying a linear change of coordinates, we may assume that $x_{p+1},\ldots ,x_n$ is a $S/I$-regular sequence. Going modulo $(x_{p+1},\dots,x_n)$, the image $I'\subseteq K[x_1,\ldots ,x_p]$ of $I$ may not be monomial, but still contains a regular sequence  of quadrics which are products of linear forms, namely the image of $f_1,\ldots ,f_g.$ So we find $J'\subseteq K[x_1,\ldots ,x_p]$ containing $(x_1^2,\ldots ,x_g^2)$ with the same Hilbert function of $I'$ by \cite[Corollary 4.3]{Ab} and \cite[Proposition 10]{CM}. Clearly  $\pd(J')\leq p -1$, but we can actually choose $J'$ such that $\pd(J')=p -1$ by \cite[Theorem 4.4]{CS}. Defining $J=J'K[x_1,\ldots ,x_n]$ we have $(x_1^2,\ldots ,x_g^2)\subseteq J$, $\pd(J)= \pd(I)$ and $\HF_I=\HF_J.$
\end{proof}

The following example shows that the above proof cannot be extended to prove EGH for all monomial ideals.

\begin{example}
The ideal $I = (x_1^2x_2, x_2^2x_3, x_1x_3^2) \subseteq K[x_1,x_2,x_3]$ does not contain a regular sequence of the form $\ell_1\ell_2\ell_3, q_1q_2q_3$, where all $\ell_i$ and $q_j$ are linear forms. 
Elementary direct computations allow one  to see that the generators are  the only products of three linear forms, which are contained  in $I$. Clearly any choice of two of them does not produce a regular sequence.
\end{example}

The following corollary is the main motivation for this note.

\begin{corollary} 
For any CM flag simplicial complex $\D$ there exists a CM balanced simplicial complex $\Gamma$ with the same $h$-vector.
\end{corollary}
\begin{proof} 
Let $g$ be the height of the Stanley-Reisner ideal $I_{\Delta}.$ By Theorem \ref{main}, there exists a CM ideal $J\subseteq S$,  containing $(x_1^2, \ldots, x_g^2)$ and  with the same Hilbert function as $I_{\D}.$ Since $J$ is unmixed, it has to be the extension to $S$ of a monomial ideal $J'\subseteq  K[x_1,\ldots ,x_g].$ Hence
$\HF_{K[x_1,\ldots ,x_g]/J'}$ equals  the $h$-vector of $\Delta.$ The CM balanced $\Gamma$ is the simplicial complex associated to the polarization of $J'$.
\end{proof}


\begin{thebibliography}{99}
\bibitem[Ab]{Ab} A. Abedelfatah, \textit{On the Eisenbud-Green-Harris conjecture},  arXiv:1212.2653v1 [math.AC], 2012.
\bibitem[BV]{BV} J. Biermann, A. Van Tuyl, \textit{Balanced vertex decomposable simplicial complexes and their $h$-vectors}, arXiv:1202.0044 [math.AC], 2012.
\bibitem[CM]{CM} G. Caviglia, D. Maclagan, 
\textit{Some cases of the Eisenbud-Green-Harris conjecture}, Math. Res. Lett. 15 (2008), no. 3, 427-433.
\bibitem[CS]{CS} G. Caviglia, E. Sbarra, \textit{The lex-plus-power inequality for local cohomology modules}, arXiv:1208.2187, 2012.
\bibitem[CL]{CL} G. F. Clements, B. Lindstr\"om, \textit{A generalization of a combinatorial theorem of Macaulay}, J. Combinatorial Theory 7, 230-238, 1969.
\bibitem[CN]{CN} D. Cook II, U. Nagel, \textit{Cohen-Macaulay Graphs and Face Vectors of Flag Complexes}, SIAM J. Discrete Math. 26, 89-101, 2012.
\bibitem[CV]{CV} A. Constantinescu, M. Varbaro, \textit{On the $h$-vectors of Cohen-Macaulay Flag Complexes}, Math. Scand., to appear.
\bibitem[EGH1]{EGH1} D. Eisenbud, M. Green and J. Harris, \textit{Higher Castelnuovo Theory}, Asterisque 218, 187-202, 1993.
\bibitem[EGH2]{EGH2} D. Eisenbud, M. Green and J. Harris, \textit{Cayley-{B}acharach theorems and conjectures},
Bull. Amer. Math. Soc. (N.S.) 33, no. 3, 295-324, 1996.
\bibitem[Fr]{Fr} A. Frohmader, \textit{Face Vectors of Flag Complexes}, Israel J. Math. 164, 153-164, 2008.
\end{thebibliography}
\end{document}